\documentclass[a4paper,11pt]{amsart}
\usepackage[utf8]{inputenc}
\usepackage{amsfonts}
\usepackage{amsmath}
\usepackage{amssymb}
\usepackage{amsthm}
\usepackage{url}
\usepackage[english]{babel}
\usepackage{paralist}
\usepackage{hyperref}
\usepackage{verbatim}

\theoremstyle{plain}
\newtheorem{thm}{Theorem}
\newtheorem{lemma}[thm]{Lemma}
\newtheorem{proposition}[thm]{Proposition}

\theoremstyle{remark}
\newtheorem{remark}{Remark}

\theoremstyle{definition}

\theoremstyle{remark}

\newcommand{\E}{\mathbb{E}}
\newcommand{\p}{\mathbb{P}}
\newcommand{\R}{\mathbb{R}}
\newcommand{\I}{\mathbb{I}}
\newcommand{\F}{\mathcal{F}}
\newcommand{\V}{\mathcal{V}}
\newcommand{\A}{\mathcal{A}}
\newcommand{\s}{\sigma}

\title{Mirror and Synchronous Couplings of Geometric Brownian Motions}
\date{}

\author{Saul D.\ Jacka}
\address{Department of Statistics, University of Warwick, UK}
\email{s.d.jacka@warwick.ac.uk}

\author{Aleksandar Mijatovi\'c}
\address{Department of Mathematics, Imperial College London, UK}
\email{a.mijatovic@imperial.ac.uk}

\author{Dejan Širaj}
\address{Department of Statistics, University of Warwick, UK}
\email{d.siraj@warwick.ac.uk}

\begin{document}

\thanks{D. Širaj acknowledges the support of
the Slovene Human Resources Development and Scholarship Fund.}



\begin{abstract}
The paper studies the question of whether the classical mirror
and synchronous couplings of two Brownian
motions minimise and maximise, respectively, the coupling time
of the corresponding geometric Brownian motions. 
We establish a characterisation of the optimality of the 
two couplings over any finite time horizon and show that,
unlike in the case of Brownian motion, the optimality fails in general 
even if the geometric Brownian motions are martingales.
On the other hand, we prove that in the cases of 
the ergodic average and the infinite time horizon criteria,
the mirror coupling and the synchronous coupling are
always optimal for general (possibly non-martingale) geometric 
Brownian motions. We show that the two couplings  
are efficient if and only if they are optimal over
a finite time horizon and give a conjectural answer 
for the efficient couplings when they are suboptimal.
\end{abstract}

\maketitle

\section{Introduction}

Let the process $B = (B_t)_{t \geq 0}$ be a fixed standard Brownian motion and
consider a standard Brownian motion
$V = (V_t)_{t \geq 0}$ 
on the same probability space. For any starting points $x, y \in \R$, define 
the \textit{coupling time}
$\tau(V)$
to be the first time the processes $x + B$ and $y + V$ meet.
It is obvious that the 
\textit{synchronous coupling}
$V=B$
maximises the coupling time 
as it makes it infinite almost surely
(assuming $x\neq y$). 
Note further that the coupling time 
$\tau(V)$
for any Brownian motion 
$V$
cannot be smaller than the first time
one of the processes 
$x + B$ 
and
$y + V$ 
reaches level
$(x+y)/2$.
In the case of the \textit{mirror coupling}
$V=-B$,
this random time actually equals 
$\tau(V)$ 
and the 
coupling inequality becomes an equality.
In particular, for any fixed $T \geq 0$, the extremal Brownian motion
in the optimisation problem,
\begin{equation}
\label{eq:fin_hor_BMs}
\text{minimise (resp. maximise)} \quad \p(\tau(V) > T) \quad \text{over all Brownian motions $V$,}
\end{equation}
is given by the mirror (resp. synchronous) coupling, uniformly
over all finite time horizons. 

It is natural to investigate the following closely related problem
for geometric Brownian motion: 
minimise the coupling time 
of the processes 
$\mathrm{d} X_t =\s_1 X_t \,\mathrm{d} B_t$
and
$\mathrm{d} Y_t(V) =\s_2 Y_t(V) \,\mathrm{d} V_t$
over all Brownian motions
$V$
on a given filtered probability space.
The aim here is to maximise the probability 
of the event that
$X$
and
$Y(V)$
couple before a given fixed time
$T$.
Since the 
processes
$X$
and
$Y(V)$
are, at any time 
$t$,
given by explicit deterministic functions of 
$B_t$
and
$V_t$
respectively,
the discussion above might suggest that mirror
coupling of 
$B$
and
$V$
should be optimal.
Furthermore, 
since 
$X$
and
$Y(V)$
are martingales, 
the Dambis-Dubins-Schwarz representation of the
difference
$X-Y(V)$
intuitively suggests that the two processes will meet as
early as possible if
the coupling of the Brownian motions 
$B$
and
$V$
is chosen so that the instantaneous volatility of
$X - Y(V)$
is as large as possible.
Equivalently put,
the minimal coupling time should be achieved by 
the Brownian motion 
$V$
which maximises
(at every moment in time) the instantaneous
quadratic variation
$ \mathrm{d}[X-Y(V)]_s
= \left((\s_1 X_s)^2 +
(\s_2 Y_s(V))^2 \right) \mathrm{d}s - 2 \s_1 X_s \s_2 Y_s(V)\,\mathrm{d}[B,V]_s.$
Since the (random) Lebesgue density of the
covariation
measure
$\mathrm{d}[B,V]_s$
on
$[0,\infty)$
is always between
$-1$
and
$1$,
it follows that 
the mirror coupling
$V = -B$
should be optimal. However, as we shall see, both of these intuitive 
arguments turn out to be false in general.

This paper investigates the problems of
minimising and maximising the coupling time
of two general (i.e. possibly non-martingale) geometric
Brownian motions (GBMs) using a finite time, 
infinite time and ergodic average criteria. 
In the finite time horizon case we study 
the analogue of 
Problem~\eqref{eq:fin_hor_BMs} for GBMs
and give a necessary and sufficient condition on the 
value function 
for the mirror (resp. synchronous)
coupling to be optimal. 
This leads to an if-and-only-if condition on the parameters of the GBMs,
which characterises the suboptimality (and hence optimality) of
the mirror (resp. synchronous) coupling 
for any finite time horizon. 
In contrast to the intuitive arguments given above, this condition implies 
that mirror (resp. synchronous) coupling can be suboptimal
in Problem~\eqref{eq:fin_hor_BMs} for GBMs
even if the geometric Brownian motions are martingales. 
This raises a natural question: is the exponential tail of 
the mirror (resp. synchronous) coupling optimal or, put differently, is the 
coupling efficient in the sense of~\cite{article4}? We show that the mirror (resp. synchronous) 
coupling is efficient if and only if it is optimal, and hence
may be inefficient. In the case where
the coupling is suboptimal, the proof of the aforementioned equivalence 
suggests the conjecture that the synchronous 
(resp. mirror) coupling is efficient in the minimisation (resp.
maximisation) problem.

The \textit{stationary} and \textit{infinite time horizon} (for some ``discount'' rate
$q>0$) problems are given as the 
analogues of Problem~\eqref{eq:fin_hor_BMs}
with 
$\p(\tau(V) > T)$
replaced by
$$
\limsup_{T \to \infty} \frac{1}{T} \int_0^T \p(\tau(V) > t) \,\mathrm{d}t
\qquad\text{ and }\qquad
\int_0^\infty \mathrm{e}^{-qt} \,\p(\tau(V) > t)\,\mathrm{d}t,$$
respectively.
It is clear that in the case of Brownian motion,
the mirror (resp. synchronous) coupling 
is optimal according to both of these criteria. 
In this paper we prove that, unlike in the finite 
time horizon case, the same holds for all (possibly non-martingale)
geometric Brownian motions. 
In particular this implies that the mirror coupling, which may be 
inefficient (i.e. has a thicker exponential tail than the optimal coupling),
nevertheless minimises both the Laplace transform of the tail 
probability for any ``discount'' rate
$q$ and its ergodic average. 
Our proofs are based on Bellman's principle.

An application in mathematical finance of the coupling problems considered in the
present paper can be described as follows. Assume that the performance of a  
portfolio manager is assessed at some fixed future time (e.g. one year from now)
with respect to a benchmark security (e.g. some equity index), which 
evolves as a geometric Brownian motion $X$.
Put differently, the remuneration of the manager depends on whether her 
portfolio, which evolves as $Y(V)$, exceeds the benchmark $X$ in normalised terms. 
Assume also that the manager's mandate stipulates
that, over the same  time horizon, her portfolio may not exceed a pre-specified  
amount of realised variance. Both of these assumptions are realistic and are
used extensively in practice, since the investor wants to beat the index but 
cannot tolerate arbitrary amounts of volatility in the mean time (e.g. investors 
like pension funds routinely stipulate such realised variance conditions). 
Imagine now a situation where the manager has a given amount of time, say 
$T$, before the evaluation of her performance, but is behind 
the benchmark by a certain amount. The question of how to trade in such a way
(a)~to minimise the probability of not catching up with the benchmark 
before 
$T$ 
and (b)~to achieve this without taking unnecessary bets which would increase 
the realised volatility of the portfolio, is precisely the question of 
the stochastic minimisation of the coupling time between 
$X$
and
$Y(V)$
(recall that the expected quadratic variation of
$Y(V)$, i.e. the realised variance of the manager's
portfolio, does not depend on the choice of Brownian motion
$V$).



The mirror coupling and the synchronous coupling of Brownian motions and related
processes have attracted much attention in the literature.
For example the classical book~\cite{Lindvall}
and paper~\cite{article6} 
introduce the mirror couplings of Brownian motions 
and diffusion processes 
(see also book~\cite{Thorisson} for the general theory of coupling).
In~\cite{article5} it is established that the mirror coupling is not the only maximal
coupling, although it is the unique maximal coupling in the family of Markovian
(also known as immersed) couplings. In \cite{article} it is proved that the tracking error of
two driftless diffusions is minimised by the synchronous coupling
of the driving Brownian motions. In \cite{article2} generalised
mirror coupling and generalised synchronous coupling of Brownian motions are introduced;
the former minimises the
coupling time and maximises the tracking error of two regime-switching
martingales, whereas the latter does the opposite. Articles \cite{article3},
\cite{article4}, and \cite{article7} discuss various applications of the mirror
coupling of reflected Brownian motions and other processes.
In particular in~\cite{article4}, the notion of 
efficiency of a 
Markovian coupling, also used in the present paper, is studied in the context of 
the spectral gap of the generator of a Markov process.

The remainder of the paper is organised as follows. Section~\ref{sec:Setting}
describes the setting and basic notation, which is used throughout. 
Section~\ref{sec:Inf_Hor} establishes the optimality of the mirror and
synchronous couplings in the infinite time horizon (Section~\ref{subsec:Inf_Time}, Theorem~\ref{thm}) 
and stationary (Section~\ref{subsec:Stationary}, Proposition~\ref{pthm}) problems.
In Section~\ref{sed:Fin_Hor}
we characterise the optimality of the mirror and synchronous 
couplings over a finite time horizon
(Section~\ref{subsec:Fin_Hor}, Theorem~\ref{neg})
and analyse the efficiency of the two couplings (Section~\ref{subsec:Efficiency}, Theorem~\ref{prop:Efficiency}).
Appendix~\ref{app} contains a well-known lemma from stochastic analysis,
which enables us to apply Bellman's principle.

\section{Setting and notation}
\label{sec:Setting}

Let $(\Omega, \F, (\F_t)_{t \geq 0}, \p)$ be a filtered probability space which
is rich enough to support a standard $(\F_t)$-Brownian motion
$B = (B_t)_{t \geq
0}$. Let
\begin{equation}
\label{eq:DEf_cV}
\V := \{ V=(V_t)_{t \geq 0};\; V \text{ is an } (\F_t)\text{-Brownian
motion with } V_0 = 0 \}
\end{equation}
be the set of all (standard) $(\F_t)$-Brownian
motions on this probability space.

Let $X = (X_t)_{t \geq 0}$ and $Y(V) = (Y_t(V))_{t \geq 0}$ be geometric
Brownian motions, satisfying stochastic differential equations
\begin{equation}
\label{gbmx} X_t = x + \int_0^t X_s \left(\s_1 \,\mathrm{d}B_s +
a_1 \,\mathrm{d}s\right)\quad \text{and}\quad
Y_t(V) = y + \int_0^t Y_s(V)
\left(\s_2 \,\mathrm{d}V_s + a_2 \,\mathrm{d}s\right).
\end{equation}
The Brownian motion
$B$
is fixed throughout and
$V$
is any element of the set
$\V$,
defined in~\eqref{eq:DEf_cV}.
We assume throughout the paper that 
\begin{equation}
\label{parameters} x, y > 0, \quad a_1, a_2 \in \R \quad \text{and} \quad \s_1,
\s_2 \in \R, \quad \text{such that} \quad \s_1 \s_2 > 0,
\end{equation}
and define the following constants 
\begin{equation}
\label{def:mu_sigma_pm}
\mu := a_2 - a_1 + \s_1^2/2 - \s_2^2/2\quad\text{ and }\quad\sigma_\pm:=\s_2\pm\s_1.
\end{equation}
Note that~\eqref{parameters} implies 
$|\sigma_+|>|\sigma_-|$.
The symbol 
$\pm$ 
denotes 
either $+$ or $-$. 
If 
$\pm$ and $\mp$ 
appear in the same expression, then they
simultaneously denote either $+$ and $-$, or $-$ and $+$.


Define the \emph{coupling time}
of the two processes in~\eqref{gbmx} as
$$\tau(V) := \inf \{t \geq 0;\; X_t = Y_t(V)\} \qquad (\inf\emptyset := \infty).$$
The random variable $\tau(V)$ 
is zero when the two processes start at the same point and
positive $\p$-a.s. otherwise.
Under mild assumptions
(e.g. if the filtration $(\F_t)_{t \geq 0}$ is right-continuous or
if all the paths of Brownian motions
$V$  and $B$ are continuous), 
$\tau(V)$ is 
$\p$-a.s. equal to 
an $(\F_t)$-stopping time. 
Furthermore, since 
$X$
and
$Y(V)$
are geometric Brownian motions, this coupling time 
can be expressed as the coupling time of a Brownian motion and a Brownian motion with drift:
$\tau(V)=\inf \{t \geq 0;\; V_t = (\sigma_1 B_t-\mu t +\log(x/y))/\sigma_2\}$.

\section{Stationary and infinite time horizon problems}
\label{sec:Inf_Hor}

\subsection{Infinite time horizon problems}
\label{subsec:Inf_Time}
For any $q> 0$, we consider the following two problems: find
$V^{\text{inf}} \in \V$ and $V^{\text{sup}} \in \V$ (if they exist) such that
\begin{equation}
\label{min} \tag{qInf} \inf_{V \in \V} \int_0^\infty \mathrm{e}^{-qt} \,
\p \big( \tau(V) > t \big) \,\mathrm{d}t = \int_0^\infty \mathrm{e}^{-qt} \,
\p \big( \tau \big( V^{\text{inf}} \big) > t \big) \,\mathrm{d}t
\end{equation}
and
\begin{equation}
\label{max} \tag{qSup} \sup_{V \in \V} \int_0^\infty \mathrm{e}^{-qt} \,
\p \big( \tau(V) > t \big) \,\mathrm{d}t = \int_0^\infty \mathrm{e}^{-qt} \,
\p \big( \tau(V^{\text{sup}}) > t \big) \,\mathrm{d}t.
\end{equation}

A simple integration by parts
yields
$\int_0^\infty \mathrm{e}^{-rt} \,\p(\tau > t) \,\mathrm{d}t =
(1 - \E( \mathrm{e}^{-r\tau} ))/r$  for
any nonnegative random variable $\tau$ and $r>0$.
Therefore Problems~\eqref{min} and~\eqref{max}
are equivalent to finding $V^{(+)} \in \V$
and
$V^{(-)} \in \V$ respectively,
such that
\begin{equation}
\label{pm} \tag{q$\pm$}
\sup_{V \in \V} \pm\E \left( \mathrm{e}^{-q\tau(V)} \right)
=\pm \E \left( \mathrm{e}^{-q\tau\left( V^{(\pm)} \right)} \right).
\end{equation}
Note also that if $e_q$ is an exponential random variable with
$\E(e_q)=1/q$,
independent of the filtration $(\F_t)_{t \geq 0}$,
then Problems~\eqref{min} and~\eqref{max} are
equivalent to minimising and maximising $\p(\tau(V) > e_q)$ over $V \in \V$,
respectively.

The following theorem holds.

\begin{thm} \label{thm}
A solution to Problem~\eqref{pm} is (for any
$q > 0$)
given by
$$V^{(\pm)} = \mp B.$$
\end{thm}

\begin{remark}
\begin{enumerate}[(i)]
\item
Observe that by Theorem~\ref{thm},
the mirror coupling
$V^{(+)} = -B$
solves Problem~\eqref{min}
and the synchronous coupling
$V^{(-)} = +B$
is the solution to Problem~\eqref{max}.
\item
Note
that the solution depends neither on the parameters in~\eqref{parameters}
nor on the discount rate
$q$.
\end{enumerate}
\end{remark}


\subsubsection{Proof of Theorem~\ref{thm}}
\label{subsec:Proof_inf_hor}

Observe that, due to the symmetry in
Problem~\eqref{pm},
we may assume without loss of generality
that
the starting points
$x,y$
in~\eqref{gbmx}--\eqref{parameters}
satisfy
$(x,y)\in D$,
where the set
$D\subset\R^2$
is given by
\begin{equation}
\label{eq:Def_D}
D:=\{(a,b);\; a\geq b>0\}.
\end{equation}
Fix $q > 0$ 
and define the following function,
closely related to the right-hand side in
Problem~\eqref{pm}:
\begin{align}
\label{eq:val_fun_pm}
\Psi^{(\pm)}(x,y) & := \E_{x,y} \left( \mathrm{e}^{-q\tau(\mp B)} \right),
\qquad (x,y)\in D.
\end{align}
The proof of Theorem~\ref{thm} is in two steps:
we first establish sufficient conditions for a function
$\Psi:D\to\R_+$
implying that
$\pm \Psi$
is equal to the right-hand side in
Problem~\eqref{pm} (Lemmas~\ref{bp} and~\ref{gen}),
and then prove that
$\Psi^{(\pm)}$
in~\eqref{eq:val_fun_pm}
satisfies these conditions
(Lemma~\ref{ls}).
Throughout the paper we denote
$\R_+ := [0,\infty)$.

For any measurable function
$\Psi:D\to\R_+$
and Brownian motion
$V\in\V$,
consider the process
$U(V,\Psi)=(U_t(V,\Psi))_{t\in[0,\infty)}$
defined by
\begin{equation}
\label{eq:Proc_U}
U_t(V,\Psi) :=
\mathrm{e}^{-q(t \wedge \tau(V))}\Psi(X_{t \wedge \tau(V)}, Y_{t \wedge \tau(V)}(V))
\end{equation}
(here and in the rest of the paper we denote
$s\wedge t := \min(s,t)$).
Then the following lemma 
(a suitable version of Bellman's principle)
holds.

\begin{lemma}
\label{bp}
Let
$\Psi:D\to\R_+$
be a bounded continuous function
satisfying
$\Psi(x,x)=1$
for all
$x>0$.
If, for every
$(x,y)\in D$,
the process
$\pm U(V,\Psi)$
is a $\p_{x,y}$-supermartingale
for all $V \in \mathcal{V}$
and
$U(\mp B,\Psi)$
is a $\p_{x,y}$-martingale,
then
$V^{(\pm)}=\mp B$
solves Problem~\eqref{pm}. 
\end{lemma}

\begin{proof}
Since 
$X_{\tau(V)}= Y_{\tau(V)}(V)$ $\p_{x,y}$-a.s. 
on the event
$\{\tau(V)<\infty\}$ 
for any $V \in \V$,
$\Psi$
is continuous and bounded,
$\Psi(x,x) = 1$ holds for any $x > 0$ 
and
$q>0$,
the supermartingale property and 
the Dominated Convergence Theorem
imply
$$\pm \E_{x,y} \left( \mathrm{e}^{-q\tau(V)} \right) =
\E_{x,y} \left(\pm U_{\tau(V)}(V,\Psi)\I_{\{\tau(V)<\infty\}} \right) \leq
\E_{x,y} \left(\pm U_{0}(V,\Psi) \right) =\pm \Psi(x,y),\quad (x,y) \in D,$$
for all $V \in \V$
($\I_{\{\cdot\}}$ denotes
the indicator of the event 
$\{\cdot\}$).
Since
$U(\mp B,\Psi)$
is a martingale, for
$V^{(\pm)}=\mp B$
this inequality becomes an equality and
the lemma follows.
\end{proof}

Our next task is to establish a verification lemma
for Problem~\eqref{pm}.
Let
$D^\circ$
be the interior (in
$\R^2$)
of the set
$D$
defined in~\eqref{eq:Def_D}.
For any twice differentiable function
$f\in\mathcal{C}^{2,2}(D^\circ)$
we define the function
$\mathcal{L}^{(\pm)}f$
by the formula
\begin{equation}
\label{eq:Def_L_pm}
\left(\mathcal{L}^{(\pm)}f\right)(x,y) :=
\left(a_1 xf_x + a_2 yf_y + \frac{1}{2}\s_1^2x^2f_{xx} +
 \frac{1}{2}\s_2^2y^2f_{yy} \mp \s_1\s_2xy f_{xy} - q f\right)(x,y),
\end{equation}
where
$(x,y)\in D^\circ$
and
$f_x,f_y,f_{xx},f_{yy}$
and
$f_{xy}$
denote the partial derivatives of
$f$.
For any function
$\Psi:D\to\R_+$,
such that
$\Psi\in\mathcal{C}^{2,2}(D^\circ)$,
and Brownian motion
$V\in\V$,
the local martingale
$M(V,\Psi)=(M_t(V,\Psi))_{t\in[0,\infty)}$,
given by
\begin{equation}
\label{eq:Def_M}
M_t(V,\Psi) := \int_0^{t \wedge \tau(V)} \mathrm{e}^{-q s}
\big( \s_1 X_s \Psi_x(X_s, Y_s(V)) \,\mathrm{d}B_s+ \s_2 Y_s(V)
\Psi_y(X_s, Y_s(V)) \,\mathrm{d}V_s \big),
\end{equation}
is well-defined.

\begin{lemma} \label{gen}
Assume the following hold:
(I) $\Psi:D\to\R_+$
is a bounded continuous  function with
$\Psi(x,x)=1$
for all $x>0$;
(II) $\Psi\in\mathcal{C}^{2,2}(D^\circ)$
and, in the interior
$D^\circ$,
$\Psi_{xy} \leq 0$ and $\mathcal{L}^{(\pm)}\Psi = 0$;
(III) $M(V,\Psi)$
is a $\p_{x,y}$-martingale
for all $(x,y)\in D$
and $V \in \V$.
Then  
for any $(x,y)\in D$,
$V \in \V$,
the process
$\pm U(V,\Psi)$,
defined in~\eqref{eq:Proc_U},
is a $\p_{x,y}$-supermartingale
and 
$U(\mp B,\Psi)$
is a $\p_{x,y}$-martingale.
\end{lemma}

\begin{proof}
The definition of
$X$
and
$Y(V)$
in~\eqref{gbmx}
and
Lemma~\ref{bm} in the Appendix imply
$\mathrm{d}[X, Y(V)]_t = C_t \s_1 X_t \s_2 Y_t(V) \,\mathrm{d}t$,
where
$C=(C_t)_{t\in[0,\infty)}$
is $(\F_t)$-adapted and
$\p(C_t\in[-1,1])=1$ for all
$t\in[0,\infty)$.
It\^o's lemma, the assumptions in Lemma~\ref{gen}
and definition~\eqref{eq:Proc_U}
of
$U(V,\Psi)$
yield
\begin{equation*}
\pm U_t(V,\Psi)  = \pm \Psi(x,y) \pm M_{t}(V,\Psi)
+\int_0^{t \wedge \tau(V)} \mathrm{e}^{-qs} \s_1 \s_2 (1 \pm C_s )X_sY_s(V) \Psi_{xy}(X_s, Y_s(V))  \,\mathrm{d}s
\end{equation*}
for all  $(x,y)\in D$ and $V \in \V$.
Since
$X$,
$Y(V)$
and
$1\pm C$
are non-negative processes and, by assumption~\eqref{parameters},
we have $\s_1\s_2>0$,
the integrand
in the representation of
$\pm U(V,\Psi)$
is non-positive,
making 
$\pm U(V,\Psi)$
a $\p_{x,y}$-supermartingale.
This representation, together with assumption~(III),
implies that 
$U(\mp B,\Psi)$
is a $\p_{x,y}$-martingale.
\end{proof}

Note the following equivalence:
\begin{equation}
\label{eq:Bad_Case}
\p_{x,y}(\tau(\mp B)=\infty)=1\quad
\text{for all $(x,y)\in D^\circ$} 
\iff
\mp = +,\> \sigma_2=\sigma_1,\> a_2\leq a_1.
\end{equation}
It is clear that 
under condition~\eqref{eq:Bad_Case}
Theorem~\ref{thm}
holds.
Lemmas~\ref{bp}
and~\ref{gen}
imply that in order to establish
Theorem~\ref{thm} in general,
it is sufficient to prove that,
when~\eqref{eq:Bad_Case} fails,
the function
$\Psi^{(\pm)}:D\to\R_+$
in~\eqref{eq:val_fun_pm}
satisfies the assumptions of Lemma~\ref{gen}.
More precisely, the following lemma holds.

\begin{lemma}
\label{ls}
Assumptions~(I)--(III) 
of Lemma~\ref{gen} hold for the function
$\Psi^{(\pm)}:D\to\R_+$
in~\eqref{eq:val_fun_pm},
if for some
$(x,y)\in D^\circ$
we have
$\p_{x,y}(\tau(\mp B)=\infty)<1$.
\end{lemma}

\begin{proof}

Under the assumption of the lemma,
the following representation holds:
\begin{equation}
\label{eq:Rerpresentaion_Val_Fun}
\Psi^{(\pm)}(x,y) =
\left( \frac{y}{x} \right)^{k_\pm}\quad
\text{for
$(x,y) \in D$,}
\end{equation}
where 
$$
k_\pm :=
\left\{ \begin{array}{ll}
-\mu/\s_\pm^2 + \sqrt{(\mu/\s_\pm^2)^2 + 2q/\s_\pm^2}, & \textrm{if $\s_\pm \neq 0$,}\\
q/\mu, & \textrm{if $\s_\pm=0$,}
\end{array} \right.
$$
and 
$\s_\pm$
and
$\mu$
are defined in~\eqref{def:mu_sigma_pm}.
Since, by assumption, the condition on the right-hand side in~\eqref{eq:Bad_Case}
is not satisfied, 
the equality
$\s_\pm=0$ implies
$\mu>0$,
making
$k_\pm$
a well-defined real number. 
Formula~\eqref{eq:Rerpresentaion_Val_Fun}
follows from the fact that
$\tau(\mp B)$
has the same law as the first-passage time
of the Brownian motion with drift
$(\s_\pm B_t + \mu t)_{t\in[0,\infty)}$
over the level
$\log(x/y)$.
The Laplace transform of this random time 
is given in \cite[p.\ 295]{Borodin}
and amounts to the right-hand side of~\eqref{eq:Rerpresentaion_Val_Fun}.

Assumption~(I)
in Lemma~\ref{gen}
follows from~\eqref{eq:Rerpresentaion_Val_Fun}. 
Furthermore
it is clear that
$\Psi^{(\pm)}\in\mathcal{C}^{2,2}(D^\circ)$.
The formula in~\eqref{eq:Rerpresentaion_Val_Fun}
and some simple calculations imply that
for $(x,y)\in D^\circ$
the following holds:
\begin{equation}
\label{eq:First_Der_Psi}
\Psi^{(\pm)}_x(x,y) = -\frac{k_\pm}{x}\,\Psi^{(\pm)}(x,y), \quad
\Psi^{(\pm)}_y(x,y) = \frac{k_\pm}{y}\,\Psi^{(\pm)}(x,y),
\end{equation}
and
$$
\Psi^{(\pm)}_{xy}(x,y) = -\frac{k_\pm^2}{xy}\,\Psi^{(\pm)}(x,y) \leq 0,\quad
\left(\mathcal{L}^{(\pm)}\Psi^{(\pm)}\right)(x,y)  = 0.
$$
Hence assumption~(II) of Lemma~\ref{gen} is also satisfied.
The equalities in~\eqref{eq:First_Der_Psi}
and the definition in~\eqref{eq:Def_M}
of the local martingale
$M(V,\Psi^{(\pm)})$
imply that the integrands in the stochastic
integrals are bounded processes and therefore
square integrable. Hence 
$M(V,\Psi^{(\pm)})$
is a $\p_{x,y}$-martingale for all $(x,y) \in D$
and $V \in \V$ and assumption~(III) of Lemma~\ref{gen}
also holds.
\end{proof}


\subsection{Stationary problems}
\label{subsec:Stationary}
Note first that Fubini's theorem and the Dominated Convergence Theorem imply
the existence of the limit:
\begin{equation}
\label{eq:Limsup_limit}
\lim_{T \to \infty} \frac{1}{T} \int_0^T \p(\tau(V) > t) \,\mathrm{d}t =
\lim_{T \to \infty} \E((\tau(V)/T) \wedge 1) = \p(\tau(V) = \infty).
\end{equation}
Hence the stationary problems from the introduction can be rephrased as:
find
$V^{\text{inf}} \in \V$ and $V^{\text{sup}} \in \V$  such that
\begin{equation}
\label{S_min} \tag{SInf} \inf_{V \in \V} 
\lim_{T \to \infty} \frac{1}{T} \int_0^T \p(\tau(V) > t) \,\mathrm{d}t =
\lim_{T \to \infty} \frac{1}{T} \int_0^T \p(\tau(V^{\text{inf}}) > t) \,\mathrm{d}t
\end{equation}
and
\begin{equation}
\label{S_max} \tag{SSup} \sup_{V \in \V} 
\lim_{T \to \infty} \frac{1}{T} \int_0^T \p(\tau(V) > t) \,\mathrm{d}t =
\lim_{T \to \infty} \frac{1}{T} \int_0^T \p(\tau(V^{\text{sup}}) > t) \,\mathrm{d}t.
\end{equation} 
A solution to these problems, independent 
of the values of the parameters of the geometric Brownian motions 
in~\eqref{gbmx}, is given in the following proposition.
Note in particular that, unlike in the finite time horizon case, 
no new phenomena arise when the ergodic average  criterion is used 
(i.e. the solution is completely analogous to the infinite time 
horizon case).
\begin{proposition} \label{pthm}
The Brownian motions 
$V^{\text{inf}} = - B$
and
$V^{\text{sup}} = B$
solve 
Problems~\eqref{S_min} and~\eqref{S_max}
respectively.
\end{proposition}

\begin{proof}
As in 
Section~\ref{subsec:Proof_inf_hor}
we may assume that, due to symmetry, 
the starting points of 
$X$
and
$Y(V)$
satisfy
$(x,y)\in D$
(see~\eqref{eq:Def_D}).
By~\eqref{gbmx} 
and 
the definition of
$\tau(V)$
in Section~\ref{sec:Setting}
we have 
\begin{equation}
\label{eq:Key_Quocient}
\tau(V)= \inf\{ t \geq 0;\; \sigma_2 V_t-\sigma_1 B_t + \mu t = \log(x/y)\},
\end{equation}
where 
$\mu$
is defined in~\eqref{def:mu_sigma_pm}
and the convention 
$\inf\emptyset =\infty$
is used.
If 
$x=y$
we have
$\tau(V)=0$
for all 
$V \in \V$ 
and Proposition~\ref{pthm}
follows. So we can assume 
$(x,y)\in D^\circ$ in the rest of the proof.

We first analyse the case 
$\mu>0$.
By~\eqref{eq:Limsup_limit}, 
Problems~\eqref{S_min} and~\eqref{S_max} are equivalent to finding $V^{(\pm)} \in \V$ such that
\begin{equation}
\label{ppm} \tag{S$\pm$}
\inf_{V \in \V} \pm \p(\tau(V) = \infty) = \pm \p(\tau(V^{(\pm)}) = \infty).
\end{equation}
The strong law of large numbers for Brownian motion 
(e.g.~\cite[p.~53]{Borodin}), representation~\eqref{eq:Key_Quocient}
and 
$\log(x/y)>0$
imply the equality
$\p_{x,y}(\tau(V) = \infty) = 0$ 
for every $V \in \V$ and Proposition~\ref{pthm} follows. 

In the case
$\mu\leq0$
we return to the formulation of 
Problems~\eqref{S_min} and~\eqref{S_max}
above. Observe that Theorem~\ref{neg}\eqref{item:b_cor} below
yields the optimal couplings that minimise and maximise the probability 
$\p(\tau(V) > t)$ for every 
$t\geq0$.
Since the couplings are independent of
$t$,
they also minimise and maximise the stationary criteria
in Problems~\eqref{S_min} and~\eqref{S_max},
which concludes the proof. 
\end{proof}

\begin{remark}
The proof of Proposition~\ref{pthm} 
relies in an obvious way on Theorem~\ref{neg}\eqref{item:b_cor} below. 
We would like to stress that there is no circularity 
in this argument since Proposition~\ref{pthm}
is not used in Section~\ref{sed:Fin_Hor}.
Stationary problems are considered in Section~\ref{sec:Inf_Hor} 
rather than later on in the paper, because the structure of the solution is the same as that of 
the infinite time horizon problems.
\end{remark}

\section{Finite time horizon problems and the efficiency of the couplings}
\label{sed:Fin_Hor}

\subsection{Finite time horizon problems}
\label{subsec:Fin_Hor}
Retain the setting and notation from Section~\ref{sec:Setting}.
For any $T > 0$, consider the following problems:
\begin{equation}
\label{Tpm} \tag{T$\pm$}
\text{find
$V^{(\pm)} \in \V$
such that }
\inf_{V \in \V} \pm
\p \big( \tau(V) > T \big)
=\pm
\p \big( \tau \big( V^{(\pm)} \big) > T \big).
\end{equation}
As in Section~\ref{sec:Inf_Hor},
we can reduce Problem~\eqref{Tpm}
to the case
where diffusions in~\eqref{gbmx} start at
$(x,y)\in D$,
where
$D$
is given in~\eqref{eq:Def_D}.
Define the set
$E:=D\times[0,T]$
and recall that the \emph{value function} for Problem~\eqref{Tpm}
is defined by
\begin{equation}
\label{eq:Def_Val_Fun}
F(x,y,t):=\inf_{V\in\V}\pm \p_{x,y} (\tau(V)>t),\qquad  (x,y,t)\in E.
\end{equation}

Based on the results in Section~\ref{sec:Inf_Hor},
one might expect that 
$\pm\Phi^{(\pm)}$,
where
\begin{align}
\label{eq:val_fun_pm_T}
\Phi^{(\pm)}(x,y,t) & := \p_{x,y} \left( \tau(\mp B)>t\right),
\qquad (x,y,t)\in E,
\end{align}
would be the value function
for Problem~\eqref{Tpm}.
In order to investigate this,
we define the function
$\A^{(\pm)}f$
for any
$f\in\mathcal{C}^{2,2,1}(E^\circ)$
($E^\circ$
is the interior of
$E$
in
$\R^3$)
by the formula
\begin{equation*}
\label{eq:Def_A_pm}
\left(\mathcal{A}^{(\pm)}f\right)(x,y,t) :=
\left(a_1 xf_x + a_2 yf_y + \frac{1}{2}\s_1^2x^2f_{xx} +
 \frac{1}{2}\s_2^2y^2f_{yy} \mp \s_1\s_2xy f_{xy} - f_t\right)(x,y,t),
\end{equation*}
where
$(x,y,t)\in E^\circ$
and
$f_x,f_y,f_t$,
etc.\ denote the partial derivatives of
$f$.
For any sufficiently smooth function
$\Phi:E\to\R_+$
and any Brownian motion
$V\in\V$,
we define
the local martingale
$N(V,\Phi)=(N_t(V,\Phi))_{t\in[0,T]}$
by
\begin{equation}
\label{eq:Def_MT}
N_t(V,\Phi) := \int_0^{t \wedge \tau(V)}
\big( \s_1 X_s \Phi_x(X_s, Y_s(V),T-s) \,\mathrm{d}B_s+ \s_2 Y_s(V)
\Phi_y(X_s, Y_s(V),T-s) \,\mathrm{d}V_s \big).
\end{equation}
The following proposition provides the key ingredient 
in the proof of Theorem~\ref{neg} below.

\begin{proposition}
\label{prop:Fin_Time}
Let a bounded function
$\Phi:E\to\R_+$
satisfy: 
(i)
$\Phi(x,x,t)=0$
for all
$x>0$ and
$t\in[0,T]$,
and
$\Phi(x,y,0)=1$
for all
$(x,y) \in D^\circ$;
(ii)
$\Phi\in\mathcal{C}^{2,2,1}(E^\circ)$
and, in the interior
$E^\circ$,
the equality
$\mathcal{A}^{(\pm)}\Phi = 0$
holds;
(iii) $N(V,\Phi)$
is a $\p_{x,y}$-martingale
for all $(x,y)\in D$
and $V \in \V$.
Then the following equivalence holds:
$$
\text{$\Phi_{xy} \geq 0$ on
$E^\circ$}
\iff
\text{$V^{(\pm)}=\mp B$  solves Problem~\eqref{Tpm}
and
$\pm\Phi$ is its value function}.
$$
\end{proposition}


\begin{proof}
\noindent ($\Rightarrow$): The proof of this implication is analogous to that
of Lemmas~\ref{bp} (Bellman's principle) 
and~\ref{gen} (submartingale property) in Section~\ref{sec:Inf_Hor}.
The process
$\pm U(V,\Phi)=(\pm U_t(V,\Phi))_{t\in[0,T]}$,
\begin{equation}
\label{eq:Sub_mart_fin_T}
U_t(V,\Phi):=\Phi(X_{t\wedge \tau(V)},Y_{t\wedge \tau(V)}(V),T-t),
\end{equation}
is a $\p_{x,y}$-submartingale
for any
$V\in\V$
and
$(x,y)\in D$ (proof as in Lemma~\ref{gen}).
For any 
$t\in[0,T]$,
the boundary conditions in assumption~(i) imply 
$$
U_t(V,\Phi) = U_{\tau(V)}(V,\Phi)=0\quad \text{$\p_{x,y}$-a.s. on $\{t\geq\tau(V)\}$.} 
$$
Hence, for any $(x,y) \in D$ and
$V \in \V$,
the submartingale property yields the inequality
\begin{eqnarray*}
\pm \p_{x,y} \left(\tau(V)>T \right) & = &
\E_{x,y} \left(\pm U_T(V,\Phi)\I_{\{\tau(V)>T\}} \right) =
\E_{x,y} \left(\pm U_T(V,\Phi)\right) \\
& \geq &  \pm \E_{x,y} U_0(V,\Phi)=\pm\Phi(x,y,T).
\end{eqnarray*}
As in Lemma~\ref{bp}, this establishes the implication
(note that, unlike Lemma~\ref{bp}, in this case we do not need, and in fact do not have,
the continuity of $\Phi$ on $E$).

\noindent ($\Leftarrow$):
Assume that there exists $(x_0, y_0, T_0)\in E^\circ$,
such that
$\Phi_{xy}(x_0, y_0, T_0) < 0$,
and that
$\pm\Phi$
is the value function of
Problem~\eqref{Tpm}. 
Bellman's principle implies that the process
$\pm U(V,\Phi)$, 
defined in~\eqref{eq:Sub_mart_fin_T},
is a $\p_{x,y}$-submartingale
for any
$V\in\V$
and
$(x,y)\in D$.
Using our assumption, we now construct a Brownian motion
$\tilde{V}^{(\pm)}\in\V$, such that
$\pm U(\tilde{V}^{(\pm)},\Phi)$
fails to be a
$\p_{x,y}$-submartingale (for any pair
$(x,y)\in D^\circ$), which will imply the proposition.

The continuity of
$\Phi_{xy}$
implies that there exists
$r>0$,
such that
$\Phi_{xy}$
is strictly negative
on the set
$K_2 := H_2 \times [T_0-2r,T_0+2r]\subset E^\circ$,
where
$H_2:=[x_0- 2r, x_0+ 2r]\times [y_0- 2r, y_0+ 2r]$.
Let
$H_1:=[x_0- r, x_0+ r]\times [y_0- r, y_0+ r]$
and define the stopping times
$\tau_1^{(\pm)}$
and
$\tau_2^{(\pm)}$
by:
\begin{equation*}
\tau_1^{(\pm)} := \inf\{ t \in[0,T]; (X_t, Y_t(\mp B)) \in H_1 \},\quad
\tau_2^{(\pm)}  := \inf\{ t \in [\tau_1,T]; (X_{t}, Y_{t}(\pm B)) \notin H_2 \}
\end{equation*}
(where $\inf \emptyset := T$).
Note that 
$\tau_1^{(\pm)} \leq \tau_2^{(\pm)} \leq T$
$\p_{x,y}$-a.s.
and
$\p_{x,y}(\tau_1^{(\pm)} <\tau_2^{(\pm)})>0$
(there is a slight abuse of notation in the definition of
$\tau_2^{(\pm)}$
as it is assumed that the process
$Y(\pm B)$,
defined in~\eqref{gbmx},
is driven by the Brownian motion
$\pm B$
as indicated, but started at the random time
$\tau_1^{(\pm)}$
and point
$Y_{\tau_1^{(\pm)}}(\mp B)$;
ditto for
$X$).

Define the process
$\tilde{V}^{(\pm)}=(\tilde{V}^{(\pm)}_t)_{t\in[0,\infty)}$ by the following formula:
$$\tilde{V}^{(\pm)}_t:=\int_0^t  \left(\mp \I_{\{s<\tau_1^{(\pm)}\}}
\pm\I_{\{ \tau_1^{(\pm)} \leq s < \tau_2^{(\pm)} \}}
\mp\I_ {\{ s \geq \tau_2^{(\pm)} \}} \right)\,
\mathrm{d}B_s,$$
where
$\I_{\{\cdot\}}$
is the indicator of the event 
$\{\cdot\}$.
Note that
$\tilde{V}^{(\pm)}$
is an
$(\F_t)$-Brownian motion by L\'evy's characterisation theorem.
It\^o's formula on the stochastic interval
$[\tau_1^{(\pm)},\tau_2^{(\pm)}]$
and assumptions~(i)--(iii) in the proposition imply the
following representation:
\begin{align*}
& \E_{x,y}\left[\pm U_{\tau_2^{(\pm)}}(\tilde{V}^{(\pm)},\Phi)\big\vert \F_{\tau_1^{(\pm)}} \right] \\
& = \pm U_{\tau_1^{(\pm)}}(\tilde{V}^{(\pm)},\Phi)
+\E_{x,y}\left[\int_{\tau_1^{(\pm)}}^{\tau_2^{(\pm)}}
2 \sigma_1\sigma_2 X_s Y_s(\tilde{V}^{(\pm)})\Phi_{xy}(X_s,Y_s(\tilde{V}^{(\pm)}),T-s)\,
\mathrm{d}s\Big\vert \F_{\tau_1^{(\pm)}} \right].
\end{align*}
The event
$\{\tau_1^{(\pm)}\in(T_0-r,T_0+r),\tau(\tilde{V}^{(\pm)})>T_0+2r\}$
has strictly positive probability and the integrand under the conditional
expectation is strictly negative on this event. We therefore find
$$
\E_{x,y}\left[\pm U_{\tau_2^{(\pm)}}(\tilde{V}^{(\pm)},\Phi)\big\vert \F_{\tau_1^{(\pm)}} \right]
<
\pm U_{\tau_1^{(\pm)}}(\tilde{V}^{(\pm)},\Phi)
\quad \text{on
$\{\tau_1^{(\pm)}\in(T_0-r,T_0+r),\tau(\tilde{V}^{(\pm)})>T_0+2r\}$}
$$
$\p_{x,y}$-a.s.
This inequality contradicts
the $\p_{x,y}$-a.s.\ inequality
$$
\E_{x,y}\left[\pm U_{\tau_2^{(\pm)}}(\tilde{V}^{(\pm)},\Phi)\big\vert \F_{\tau_1^{(\pm)}} \right]
\geq
\pm U_{\tau_1^{(\pm)}}(\tilde{V}^{(\pm)},\Phi),$$
which  follows from the optional sampling theorem applied to
the bounded $\p_{x,y}$-submartingale
$U(\tilde{V}^{(\pm)},\Phi)$.
This concludes the proof.
\end{proof}

We will now apply Proposition~\ref{prop:Fin_Time}
to study the question of whether
$\pm\Phi^{(\pm)}$,
defined in~\eqref{eq:val_fun_pm_T},
is the value function for Problem~\eqref{Tpm}.

\begin{lemma}
\label{der}
Recall that 
$\mu$
and
$\s_\pm$
are given in~\eqref{def:mu_sigma_pm}
and assume
$\s_\pm\neq0$.
Then, assumptions (i)--(iii) of Proposition~\ref{prop:Fin_Time}
hold for the function
$\Phi^{(\pm)}$
defined in~\eqref{eq:val_fun_pm_T}.
Furthermore, we have
\begin{equation*}
\Phi^{(\pm)}_{xy}(x,y,t) = \frac{2\log \left(x/y\right) - 4\mu t}{xy(|\s_\pm| \sqrt{t})^3}
\,n\left( \frac{\log \left(x/y\right) - \mu t}{|\s_\pm| \sqrt{t}} \right) +
\frac{4\mu^2}{xy\s_\pm^4} \left(\frac{x}{y}\right)^{ 2\mu/\s_\pm^2 }
N\left( \frac{-\log\left(x/y\right) - \mu t}{|\s_\pm| \sqrt{t}} \right)
\end{equation*}
for all $(x,y)\in D^\circ$ and $t > 0$,
where
$N(\cdot)$
is the standard normal distribution function and
$n(\cdot)$
is its density.
\end{lemma}

\begin{proof}
The explicit formula for the distribution of the running maximum of a Brownian
motion with drift (see e.g.~\cite[p.~250]{Borodin})
yields the following representation of the function in~\eqref{eq:val_fun_pm_T}:
\begin{equation}
\label{eq:Phi_Rep}
\Phi^{(\pm)}(x,y,t)= h^{(\pm)} \left( \log \left(x/y 
\right) ,t \right)
\quad\text{for}\quad (x,y)\in D,
\end{equation}
where,
for any
$ z\geq0$ and $s>0$,
we define
\begin{eqnarray}
\label{eq:h_rep}
h^{(\pm)}(z,s) := 
N \left( \frac{z - \mu s}{|\s_\pm| \sqrt{s}} \right) -
\exp \left( \frac{2\mu z}{\s_\pm^2} \right)
N \left( \frac{-z - \mu s}{|\s_\pm| \sqrt{s}} \right).
\end{eqnarray}
Simple (but tedious) calculations using this representation
yield the properties required in assumptions~(i)--(iii)
of Proposition~\ref{prop:Fin_Time}. Indeed,
note that the partial derivatives 
$h^{(\pm)}_z$,
$h^{(\pm)}_{zz}$
and
$h^{(\pm)}_s$,
take the form (recall $n'(x)=-xn(x)$):
\begin{eqnarray*}
h^{(\pm)}_z(z,s) & = &\frac{2}{|\sigma_\pm|\sqrt{s}}\> n \left( \frac{z - \mu s}{|\s_\pm| \sqrt{s}} \right) 
- \frac{2\mu}{\sigma_\pm^2} \exp\left( \frac{2\mu z}{\s_\pm^2} \right) N \left( \frac{-z - \mu s}{|\s_\pm| \sqrt{s}} \right);\\
h^{(\pm)}_{zz}(z,s) & = &\frac{4s\mu-2z}{(|\sigma_\pm|\sqrt{s})^3}\> n \left( \frac{z - \mu s}{|\s_\pm| \sqrt{s}} \right) 
- \frac{4\mu^2}{\sigma_\pm^4} \exp\left( \frac{2\mu z}{\s_\pm^2} \right) N \left( \frac{-z - \mu s}{|\s_\pm| \sqrt{s}} \right);\\
h^{(\pm)}_s(z,s) & = &- \frac{z}{|\sigma_\pm|s^{3/2}}\> n \left( \frac{z - \mu s}{|\s_\pm| \sqrt{s}} \right). 
\end{eqnarray*}
These formulae and the representation in~\eqref{eq:Phi_Rep}
imply the formula for 
$\Phi^{(\pm)}_{xy}(x,y,t)$,
as well as assumptions (i) and (ii) of Proposition~\ref{prop:Fin_Time}.
The martingale property of the process in~\eqref{eq:Def_MT} 
(i.e. assumption  (iii) in Proposition~\ref{prop:Fin_Time})
follows by It\^o's isometry from the fact that both functions
$$x\Phi^{(\pm)}_x(x,y,t)=h^{(\pm)}_z(\log(x/y),t)
\quad\text{and}\quad
y\Phi^{(\pm)}_y(x,y,t)=-h^{(\pm)}_z(\log(x/y),t)$$
are bounded on $E$. This completes the proof of the lemma.
\end{proof}

We are now ready to prove that
the mirror
(resp.\ synchronous)
coupling of the driving Brownian motions
in~\eqref{gbmx}
is not
necessarily optimal in
Problem~(T$+$) (resp.\ (T$-$)).
In Theorem~\ref{neg},
we give a necessary and sufficient condition
for the function
$\pm\Phi^{(\pm)}$,
defined~\eqref{eq:val_fun_pm_T},
to be the value function for Problem~\eqref{Tpm}.


\begin{thm}
\label{neg}
Recall that 
$\mu$
and
$\sigma_\pm$
are given in~\eqref{def:mu_sigma_pm}.
Then the following holds for any positive time horizon
and distinct starting points: 
\begin{enumerate}[(a)]
\item 
If 
$\mu>0$
and
$\sigma_\pm\neq0$,
then 
$V^{(\pm)}=\mp B$  does NOT solve Problem~\eqref{Tpm}.
\label{item:a_cor}
\item If
$\mu\leq0$,
then 
$V^{(\pm)}=\mp B$ solves
Problem~\eqref{Tpm}
with the value function 
$\pm\Phi^{(\pm)}$
in~\eqref{eq:val_fun_pm_T}.
\label{item:b_cor}
\end{enumerate}
\end{thm}

\begin{remark}
\begin{enumerate}[(i)]
\item Note that under the assumptions of Theorem~\ref{neg}\eqref{item:a_cor}, 
the mirror and synchronous couplings are suboptimal in Problems~(T$+$)
and~(T$-$) respectively. Furthermore, 
if
$\pm=+$,
then 
$\sigma_\pm>0$
and hence the optimality of the mirror coupling can fail 
even if the laws of 
$X$
and
$Y(V)$
are equivalent 
(i.e.
$\sigma_1=\sigma_2$) 
for all
$V\in\V$. 
\item In the case 
$\mu>0$
and
$\sigma_\pm=0$
we have
$\pm=-$,
$\sigma_1=\sigma_2$
and 
$\Phi^{(-)}(x,y,t)=\I_{\{t\mu<\log(x/y)\}}$
for all
$(x,y)\in D^\circ$,
$t\in[0,T]$ (recall~\eqref{eq:Key_Quocient}),
which implies that 
the synchronous coupling is suboptimal if and only if
$T\geq
\log(x/y)/\mu.$
\end{enumerate}
\end{remark}

\begin{proof}
\noindent~\eqref{item:a_cor} 
By Proposition~\ref{prop:Fin_Time} 
it suffices to show that
for any fixed
$t>0$,
there exists
$(x,y)\in D^\circ$
(see~\eqref{eq:Def_D} for the definition of 
$D$)
such that 
$\Phi^{(\pm)}_{xy}(x,y,t)< 0$.

Define 
$z:=\log(x/y)/(|\sigma_\pm|\sqrt{t})>0$
and
$\alpha:=\mu\sqrt{t}/|\sigma_\pm|>0$.
Note that, since we are allowed to choose  
the point 
$(x,y)\in D^\circ$
arbitrarily close to the diagonal half-line in the boundary of 
$D$,
a Taylor expansion of order one 
of
$z\mapsto n(z-\alpha)$
and 
$z\mapsto N(-z-\alpha)$
around 
$z=0$,
the representation  of 
$\Phi^{(\pm)}_{xy}$
in Lemma~\ref{der}
and 
the inequality 
\begin{equation}
\label{eq:final_inequality}
\alpha N(-\alpha)<n(-\alpha)
\end{equation}
imply that 
$\Phi^{(\pm)}_{xy}(x,y,t)< 0$
for some 
$(x,y)\in D^\circ$.
To check~\eqref{eq:final_inequality},
note that 
$un(u)=-n'(u)$
and
$$
\alpha N(-\alpha) = \int_\alpha^\infty \alpha n(u)\,\mathrm{d} u<
\int_\alpha^\infty un(u)\,\mathrm{d} u =n(-\alpha).
$$

\noindent~\eqref{item:b_cor}
Assume first 
$\sigma_\pm\neq0$.
Then the representation of
$\Phi^{(\pm)}_{xy}$
in Lemma~\ref{der}
and the assumption 
$\mu\leq 0$
imply 
$\Phi_{x,y}\geq0$
on
$E^\circ$.
Hence
Proposition~\ref{prop:Fin_Time}
yields the theorem. 
If 
$\sigma_\pm=0$,
we have 
$\pm=-$,
$\sigma_1=\sigma_2$
and, by~\eqref{eq:Key_Quocient}, 
it follows
$\Phi^{(-)}(x,y,t)=1$
for all
$(x,y)\in D^\circ$,
$t\in[0,T]$.
Hence 
$-\Phi^{(-)}$
is the value function for Problem~(T$-$)
and the theorem follows. 
\end{proof}

\subsection{Efficiency of the mirror and synchronous couplings}
\label{subsec:Efficiency}
In this section we examine further the (lack of) optimality of the mirror and synchronous couplings
characterised by the assumptions of Theorem~\ref{neg}\eqref{item:a_cor}.  
Since in this case the two couplings do not
minimise and maximise (respectively) the coupling times 
of the geometric Brownian motions in~\eqref{gbmx} over finite
time horizons, but are nonetheless optimal both over the infinite time horizon
(Section~\ref{subsec:Inf_Time}) and for the stationary
criterion (Section~\ref{subsec:Stationary}),
it is natural to analyse whether the two couplings 
are efficient.  A coupling
$V\in\V$
is \textit{(exponentially) efficient}
(for some
$(x,y)\in D^\circ$)
in Problem~\eqref{Tpm}
if the rate of the exponential decay of the tail 
of its coupling time 
is the same as 
the exponential decay of the value function $F$, defined in~\eqref{eq:Def_Val_Fun},
in the following sense:
\begin{eqnarray}
\label{eq:Exp_Eff}
\pm \liminf_{t\to\infty} \frac{1}{t} \log\p_{x,y}\left(\tau(V)>t\right) & \leq & \pm \liminf_{t\to\infty} \frac{1}{t} \log \left(\pm F(x,y,t)\right).
\end{eqnarray}
Note that by~\eqref{eq:Def_Val_Fun} the opposite 
inequality in~\eqref{eq:Exp_Eff} holds for any coupling
$V$.
Hence we could have defined exponential efficiency by requiring equality in~\eqref{eq:Exp_Eff}.
Furthermore, if the limits on both sides of~\eqref{eq:Exp_Eff} exist,
the definition of the exponential efficiency 
in Problem~\eqref{Tpm}
can be further simplified to
$$
\lim_{t\to\infty} \frac{1}{t} \log\p_{x,y}\left(\tau(V)>t\right) = \lim_{t\to\infty} \frac{1}{t} \log \left(\pm F(x,y,t)\right).
$$

It is clear that if a coupling solves Problem~\eqref{Tpm}
for all time horizons
$T>0$
and does not depend on
$T$,
then it is also efficient according to the definition in~\eqref{eq:Exp_Eff}. 
Hence, by Theorem~\ref{neg}\eqref{item:b_cor},
the mirror and synchronous couplings are efficient 
if 
$\mu\leq0$.
However, the following statement holds. 

\begin{thm}
\label{prop:Efficiency}
If
$\mu>0$,
the mirror and synchronous couplings are 
NOT efficient 
(for any $(x,y)\in D^\circ$)
in Problems~$(\mathrm{T}+)$ and~$(\mathrm{T}-)$  respectively. 
\end{thm}

\begin{remark}
We thank the referee for noting that 
Theorem~\ref{prop:Efficiency} holds
without assuming
$\sigma_\pm\neq0$.
\end{remark}

\begin{proof}
The following bounds hold for the standard normal distribution function
$N(\cdot)$,
\begin{eqnarray}
\label{eq:NOrmal_Bounds}
-\frac{z}{1+z^2}\,n(z) \leq 
N(z)\leq 
-z^{-1} n(z) \qquad\text{for any $z<0$, where $n=N'$.}
\end{eqnarray}
The first inequality follows from the identity 
$\int_r^\infty (1+y^{-2})\mathrm{e}^{-y^2/2}\,\mathrm{d}y=r^{-1}\mathrm{e}^{-r^2/2}$
for all
$r>0$,
and the second is given in~\eqref{eq:final_inequality}.

Assume now that 
$\sigma_\pm\neq0$.
Let
$$Z(t):=\frac{\log(x/y) - \mu t}{|\s_\pm| \sqrt{t}} \qquad\text{and}\qquad
\widehat Z(t):=\frac{-\log(x/y) - \mu t}{|\s_\pm| \sqrt{t}},
$$ 
and note that for all large $t>0$
we have
$\widehat Z(t)<Z(t)<0$
and the equality 
\begin{eqnarray}
\label{eq:idnetityNormal}
n\left(Z(t) \right) = 
\left( \frac{x}{y} \right)^{2\mu/\sigma_\pm^2}
n \left(\widehat Z(t)\right)
\end{eqnarray}
holds.
The representations in~\eqref{eq:Phi_Rep} and~\eqref{eq:h_rep}
imply
\begin{eqnarray}
\label{eq:main_Estimate}
\Phi^{(\pm)}(x,y,t) & = & 
N(Z(t)) \left[1- \left( \frac{x}{y} \right)^{2\mu/\sigma_\pm^2}\frac{N(\widehat Z(t))}{N(Z(t))}\right].
\end{eqnarray}
The inequalities in~\eqref{eq:NOrmal_Bounds}
imply the limit
\begin{equation}
\label{lim:NZ}
\lim_{t\to\infty}\frac{1}{t}\log N(Z(t))= -\frac{\mu^2}{2\s_\pm^2}.
\end{equation}
In order to deal with the second factor on the right-hand side of~\eqref{eq:main_Estimate},
we note the following inequalities,
$$
1- \left( \frac{x}{y} \right)^{2\mu/\sigma_\pm^2}\frac{N(\widehat Z(t))}{N(Z(t))}\geq
1+ (1+Z^2(t)) \frac{N(\widehat Z(t))}{n(\widehat Z(t))Z(t)} 
\geq 1- \frac{1+Z^2(t)}{\widehat Z(t)Z(t)},
$$
which are a consequence of two applications of the second inequality in~\eqref{eq:NOrmal_Bounds}
and identity~\eqref{eq:idnetityNormal}.
Let the assumption 
\begin{equation}
\label{eq:Final_as}
\log(x/y)>\frac{\sigma_+^2}{2\mu}
\end{equation}
hold.
Then the following inequality is satisfied 
$$
1- \frac{1+Z^2(t)}{\widehat Z(t)Z(t)} = \frac{t^{-1}(2\mu\log(x/y)-\sigma_\pm^2)-t^{-2}2\log^2(x/y)}{\mu^2 -t^{-2}\log^2(x/y)}>0\qquad \text{for all large $t>0$,}
$$
and the limit holds:
\begin{equation}
\label{eq:limZ}
\lim_{t\to\infty}\frac{1}{t}\log \left(1- \frac{1+Z^2(t)}{\widehat Z(t)Z(t)}\right)=0.
\end{equation}
By~\eqref{eq:main_Estimate}
we have
$$
N(Z(t)) \left[1- \frac{1+Z^2(t)}{\widehat Z(t)Z(t)}\right]\leq
\Phi^{(\pm)}(x,y,t) \leq 
N(Z(t)). 
$$
If the starting points 
$x,y$
satisfy~\eqref{eq:Final_as},
then~\eqref{lim:NZ},~\eqref{eq:limZ},
the inequalities in the line above
and the fact that $\log$ is increasing
imply the limit:
$\lim_{t \to \infty} \frac{1}{t} \log\Phi^{(\pm)}(x,y,t) = -\frac{\mu^2}{2\s_\pm^2}$.

In order to see that this limit holds without assumption~\eqref{eq:Final_as},
i.e. for $(x,y)\in D$
such that 
$\log(x/y)\in(0,\sigma_\pm^2/(2\mu)]$,
define a Brownian motion with drift
$W$,
started from
$0$,
and its first-passage time $T(z)$:
$$
W_t:=\mp\sigma_\pm B_t+\mu t,\quad t\geq0,\qquad\text{and}\qquad
T(z):=\inf\{t\geq0:W_t=z\},\quad z\in\R,
$$
and note that 
$\p_{x,y}(\tau(\mp B)>t)=\p(T(\log(x/y))>t)$
holds for any
$(x,y)\in D$
(cf.~\eqref{eq:Key_Quocient}).
Fix
$(x,y)\in D$
that violates assumption~\eqref{eq:Final_as} 
and pick
$\alpha_0<0$
and 
$(x_0,y_0)\in D^\circ$
such that the following holds:
$$
\log(x_0/y_0)=\log(x/y)-\alpha_0>\frac{\sigma_+^2}{2\mu}.
$$
Denote the constant
$q:=\p(W_1<\alpha_0,T(\log(x/y))>1)$,
which clearly satisfies
$q\in(0,1)$.
The Markov property of $W$
at time
$1$
yields the following inequalities for all $t>1$:
\begin{eqnarray*}
\p_{x,y}(\tau(\mp B)>t) & = & \p(T(\log(x/y))>t)\\
& \geq & q 
\p(T(\log(x/y)-\alpha_0)>t-1)\\
& > & q \p(T(\log(x/y)-\alpha_0)>t)=q \p_{x_0,y_0}(\tau(\mp B)>t).
\end{eqnarray*}
Since~\eqref{eq:main_Estimate}
implies the bound
$\p_{x,y}(\tau(\mp B)>t)\leq N(Z(t))$
for any  
$(x,y)\in D^\circ$,
the following limits hold 
\begin{equation}
\label{eq:Calc_pm}
\lim_{t \to \infty} \frac{1}{t} \log\p_{x,y}(\tau(\mp B) > t) = 
\lim_{t \to \infty} \frac{1}{t} \log\Phi^{(\pm)}(x,y,t) = -\frac{\mu^2}{2\s_\pm^2},
\end{equation}
by the inequality above, our analysis under assumption~\eqref{eq:Final_as}
and the limit in~\eqref{lim:NZ}.


Definition~\eqref{def:mu_sigma_pm}
and 
assumption
$\sigma_\pm\neq0$
imply
$|\sigma_+|>|\sigma_-|>0$
and hence
$\mu/(2\sigma_+^2)<\mu/(2\sigma_-^2)$.
The mirror coupling is therefore not efficient for Problem~(T$+$)
since it has a strictly thicker exponential tail than 
the synchronous 
coupling. Likewise, 
the synchronous coupling is not efficient for Problem~(T$-$),
which requires the thickest possible exponential tail among all couplings,
since it has a thinner tail than the mirror coupling. 

In the case 
$\sigma_\pm=0$
we have
$\sigma_1=\sigma_2$
and,
by~\eqref{eq:Key_Quocient},
$\tau(B)=\log(x/y)/\mu$.
Hence 
$\p_{x,y}\left(\tau(B)>t\right)=0$ 
for all 
$t\geq\log(x/y)/\mu$.
Since the limit in~\eqref{eq:Calc_pm}
still holds
for $\pm=+$ 
(note that 
$|\sigma_+|>0$),
we obtain the inequality:
$$
\lim_{t\to\infty} \frac{1}{t} \log\p_{x,y}\left(\tau(B)>t\right)=-\infty <  
-\frac{\mu^2}{2\s_+^2} = \lim_{t\to\infty} \frac{1}{t} \log\p_{x,y}\left(\tau(-B)>t\right). 
$$
This inequality and definition~\eqref{eq:Exp_Eff}
imply that the mirror (resp. synchronous) coupling is not efficient for 
Problem~(T$+$) (resp.~(T$-$)). 
\end{proof}

\begin{remark}
It is the presence of the positive drift
$\mu>0$
that makes the mirror coupling suboptimal 
in Problem~(T$+$) (see Theorem~\ref{neg}).
The proof of Theorem~\ref{prop:Efficiency} 
suggests that if the drift
$\mu$ 
is strictly positive and the time horizon $T$ is large, 
it is in fact better (according to the exponential tail criterion) to use synchronous coupling. 
This naturally leads to the following conjecture: 
\begin{flushleft}
\textit{If $\mu>0$, 
the synchronous (resp. mirror) coupling is efficient 
in Problem~$(\mathrm{T}+)$ (resp.~$(\mathrm{T}-)$).}
\end{flushleft}
\end{remark}

\appendix
\section{Family of Brownian motions on a filtered probability space}
\label{app}
Recall that $\V$
is defined in~\eqref{eq:DEf_cV}.
See e.g.~\cite[Lemma~2.1]{article2} for the proof of Lemma~\ref{bm}.

\begin{lemma}
\label{bm} For any Brownian motion $V \in \mathcal{V}$, there exists
an $(\mathcal{F}_t)$-Brownian motion $W \in \mathcal{V}$ and
a process $C = (C_t)_{t \geq 0}$,
such that $B$ and $W$ are independent, $C$ is progressively measurable
with $-1 \leq C_t \leq 1$ for all $t \geq 0$ $\p$-a.s., and the following
representation holds:
$$V_t = \int_0^t C_s \,\mathrm{d}B_s + \int_0^t \sqrt{1 - C_s^2} \,\mathrm{d}W_s.$$
\end{lemma}

\begin{remark}
The proof of this lemma requires the existence of a Brownian motion
$B^{\perp} \in \V$
that is independent of $B$. If our probability space did not support such a
Brownian motion, we could enlarge it, which would only increase the set $\V$.
Since the optimal Brownian motions in Theorems~\ref{thm} and~\ref{neg}\eqref{item:b_cor}
are constructed from $B$
alone, they would also have to be optimal in the original problem. Therefore we
can assume that $B^{\perp}$ exists.
\end{remark}

\end{document}